\documentclass[11pt]{amsart}

\usepackage[a4paper,margin=1in]{geometry}
\usepackage[T1]{fontenc}
\usepackage[utf8]{inputenc}
\usepackage{lmodern}
\usepackage{microtype}
\usepackage{amsmath,amssymb,amsthm,mathtools,mathrsfs,bm}
\usepackage{enumitem}
\usepackage{booktabs,array,longtable}
\usepackage{xcolor}
\usepackage[colorlinks=true,linkcolor=blue!50!black,citecolor=blue!50!black,urlcolor=blue!50!black]{hyperref}
\usepackage[nameinlink,capitalize]{cleveref}

\newtheorem{theorem}{Theorem}[section]
\newtheorem{proposition}[theorem]{Proposition}
\newtheorem{lemma}[theorem]{Lemma}
\newtheorem{corollary}[theorem]{Corollary}
\newtheorem{definition}[theorem]{Definition}

\newtheorem{remark}[theorem]{Remark}

\newcommand{\R}{\mathbb R}

\newcommand{\Tail}{\operatorname{Tail}}

\title[Boundary hyperbolic packets]{Boundary Hyperbolic Packets in Axisymmetric Euler Flow}

\author{Rishad Shahmurov}
\address{Cellular Products Research and Development, Roswell, Georgia 30075, USA}
\email{}


\begin{document}

\begin{abstract}
We study hyperbolic side-wall packets for the axisymmetric Euler equations with swirl in a periodic cylinder.  In the exact odd class, $\Gamma=ru^\theta$ and $G=\omega^\theta/r$ satisfy $D_t\Gamma=0$ and $D_tG=r^{-4}\partial_z(\Gamma^2)$, and the symmetry fixes the boundary point at which the relevant swirl and vorticity gradients are measured.  We derive the five-dimensional side-wall Green-kernel expansion, its leading signed hyperbolic kernel, packet--core coupling, off-diagonal shear estimates, anisotropy and source-curvature identities, and a maximal conic-score formulation.  Under a common admissible packet bootstrap these estimates imply the Dini system $D^+\mathfrak M\ge cB^2$, $B'\ge c\mathfrak M B$, and hence finite-time blow-up of the comparison amplitudes.  The result is a conditional amplification criterion: it does not prove that one smooth Euler trajectory preserves all packet-dominance, tail, shear, anisotropy, and validity hypotheses until the comparison blow-up time.  Establishing or refuting that persistence is the remaining problem for this mechanism.
\end{abstract}

\maketitle

\tableofcontents

\section{Introduction}

The three-dimensional incompressible Euler equations
\begin{equation}\label{eq:euler}
\partial_tu+(u\cdot\nabla)u+\nabla p=0,
\qquad
\nabla\cdot u=0,
\end{equation}
are the inviscid model for an ideal incompressible fluid.  Whether smooth three-dimensional Euler solutions can develop a finite-time singularity from smooth finite-energy data is one of the central open problems in mathematical fluid dynamics.  The Beale--Kato--Majda criterion \cite{BKM1984} identifies the time-integrability of the maximum vorticity as a decisive breakdown criterion.  This makes the geometry of vorticity amplification, vortex stretching, and nonlocal velocity recovery central to any proposed blow-up mechanism.

The axisymmetric-with-swirl class is one of the most important reduced three-dimensional geometries in which the Euler equations still retain a genuinely three-dimensional vortex-stretching mechanism.  In cylindrical coordinates \((r,\theta,z)\), an axisymmetric velocity with swirl has the form
\begin{equation}\label{eq:axisym_velocity}
u(r,z,t)=u^r(r,z,t)e_r+u^\theta(r,z,t)e_\theta+u^z(r,z,t)e_z.
\end{equation}
The angular momentum
\begin{equation}\label{eq:Gamma_def}
\Gamma=r u^\theta
\end{equation}
is transported by the meridional flow, while the vorticity-ratio variable
\begin{equation}\label{eq:G_def}
G=\frac{\omega^\theta}{r}
\end{equation}
is forced by the axial derivative of \(\Gamma^2\).  Precisely,
\begin{equation}\label{eq:Gamma_G_equations_intro}
D_t\Gamma=0,
\qquad
D_tG=r^{-4}\partial_z(\Gamma^2),
\end{equation}
where
\[
D_t=\partial_t+u^r\partial_r+u^z\partial_z.
\]
Thus the swirl is transported, but its gradient can force growth of \(G\).  This is the fundamental source mechanism studied in this manuscript.

The boundary hyperbolic scenario for axisymmetric Euler with swirl was brought into sharp focus by the numerical work of Luo and Hou \cite{LuoHouPNAS2014,LuoHouMMS2014}, who studied the three-dimensional axisymmetric Euler equations in a periodic cylinder with solid boundary and reported strong numerical evidence for a potential finite-time singularity at the boundary.  Related one-dimensional and reduced models inspired by this mechanism have been studied rigorously, including finite-time blow-up for Hou--Luo-type models \cite{ChoiHouKiselevLuoSverakYao2017,ChenHouHuang2022}.  The broader landscape also includes Elgindi's finite-time singularity formation for \(C^{1,\alpha}\) Euler solutions \cite{Elgindi2021}, boundary blow-up constructions in related settings \cite{ChenHou2021}, and small-scale creation mechanisms for incompressible Euler in bounded domains \cite{KiselevSverak2014}.

Chen and Hou established a computer-assisted boundary singularity result for smooth three-dimensional Euler flow by a different nearly self-similar framework \cite{ChenHou2025}.  The goal here is different: we isolate a direct kernel--packet mechanism in the original axisymmetric variables and formulate the precise persistence assumptions needed for its amplification argument.  The central object is an active boundary packet near the side wall of the periodic cylinder
\begin{equation}\label{eq:cylinder}
\Omega=\{(r,\theta,z):0<r<1,\ z\in\mathbb T\},
\end{equation}
with impermeable boundary condition
\begin{equation}\label{eq:impermeable}
u^r(1,z,t)=0.
\end{equation}
Near the side wall we set
\begin{equation}\label{eq:flat_coords}
x=1-r,
\qquad
y=z.
\end{equation}
The candidate singular point is \((x,y)=(0,0)\), corresponding to \((r,z)=(1,0)\).  The active packet is
\begin{equation}\label{eq:packet_intro}
\mathcal P_\lambda=\{0<x<\lambda,
\ |y|<\lambda\}.
\end{equation}

The key elliptic fact is that the axisymmetric recovery can be written in a five-dimensional lifted form.  If
\begin{equation}\label{eq:phi_recovery_intro}
u^r=-r\partial_z\phi,
\qquad
u^z=2\phi+r\partial_r\phi,
\end{equation}
then
\begin{equation}\label{eq:5d_recovery_intro}
-\Delta_5\phi=G,
\qquad
\Delta_5=\partial_r^2+\frac3r\partial_r+\partial_z^2.
\end{equation}
The side-wall boundary condition implies that \(\phi\) is constant on \(r=1\), and after normalization we take \(\phi|_{r=1}=0\).  The boundary compression rate is
\begin{equation}\label{eq:sigma_intro}
\sigma(t)=-\partial_z u^z(1,0,t)=-\partial_z\partial_r\phi(1,0,t).
\end{equation}
The side-wall Green-kernel expansion gives the leading effective two-variable kernel
\begin{equation}\label{eq:K0_intro}
K_0(x,y)=C_0\frac{xy}{(x^2+y^2)^2},
\qquad C_0>0.
\end{equation}
This kernel is positive for \(x>0,y>0\) and odd in \(y\).  Hence, for an odd sign-coherent packet
\[
G(x,-y)=-G(x,y),
\qquad
G(x,y)\ge0\quad(y>0),
\]
the local packet contribution creates positive hyperbolic compression.

The rest of the paper turns this observation into a theorem package.  The package contains:
\begin{enumerate}[label=(\roman*)]
\item a side-wall Green-kernel sign/parity expansion;
\item a boundary Biot--Savart sign lemma;
\item tail dominance or packet reselection;
\item packet--core coupling between the weighted hyperbolic mass and the odd core slope \(\partial_zG(1,0,t)\);
\item off-diagonal shear suppression in the inner core;
\item radial-swirl anisotropy persistence;
\item differentiated \(G\)-source amplification;
\item an odd-symmetry ODE comparison for normalized amplitudes.
\end{enumerate}
The resulting conditional engine is
\begin{equation}\label{eq:ODE_engine_intro}
A'(t)\gtrsim B(t)^2,
\qquad
B'(t)\gtrsim A(t)B(t),
\end{equation}
where
\begin{equation}\label{eq:A_B_intro}
A(t)=\lambda(t)\partial_zG(1,0,t),
\qquad
B(t)=\lambda(t)^{1/2}\partial_z\Gamma(1,0,t).
\end{equation}
This ODE comparison blows up in finite time.  Therefore the remaining task is not conceptual but invariant: one must prove, or disprove, that the packet hypotheses persist for smooth axisymmetric Euler data.

\subsection{Main conditional criterion}
The theorem proved below is conditional on the persistence of one common admissible packet regime.  The hypotheses are stated entirely in terms of the Euler solution and the side-wall recovery estimates.

\medskip
\noindent\textbf{Boundary-packet amplification criterion.}
\begin{theorem}\label{thm:intro_main_exact_odd}
Let $u$ be a smooth exact-odd axisymmetric Euler solution with swirl in the periodic cylinder.  Assume on a time interval that the side-wall Green-kernel expansion has the stated uniform remainder bounds, a narrow-diagonal packet remains dominant, the packet--core coupling and off-diagonal shear estimates remain valid, the anisotropy/source-curvature bounds hold, and the packet stays inside the region of validity of these estimates.  Then the maximal conic score $\mathfrak M_{\delta_c}$ and the normalized swirl-gradient amplitude $B$ satisfy
\[
D^+\mathfrak M_{\delta_c}\ge c_1B^2,
\qquad
B'\ge c_2\mathfrak M_{\delta_c}B.
\]
If both quantities are initially positive, the comparison system blows up in finite time.  Consequently the smooth Euler solution cannot remain in this admissible packet regime up to that comparison time unless a classical continuation norm becomes unbounded.
\end{theorem}

\section{Physical and geometric interpretation}

The mechanism studied here is a boundary version of hyperbolic small-scale creation.  The physical loop is
\begin{equation}\label{eq:physical_loop}
\boxed{\begin{gathered}
\text{odd }G\text{ packet}\Rightarrow\text{boundary compression}
\Rightarrow\partial_z\Gamma\text{ steepening}\\
\Rightarrow\partial_z(\Gamma^2)\text{ source}
\Rightarrow\text{larger }G\text{ packet}.
\end{gathered}}
\end{equation}
The side wall is essential.  The impermeability condition \(u^r=0\) allows a hyperbolic stagnation structure on the wall:
\begin{equation}\label{eq:hyperbolic_velocity}
u^z(1,y,t)\approx -\sigma(t)y,
\qquad
\sigma(t)>0.
\end{equation}
This compresses material points toward \(y=0\) along the boundary.  Since \(\Gamma\) is transported, axial compression amplifies \(\partial_z\Gamma\).  The amplified swirl gradient then feeds \(G\) through
\[
D_tG=r^{-4}\partial_z(\Gamma^2).
\]
The nonlocal Biot--Savart law converts the odd \(G\)-packet back into the compression \(\sigma\).  Thus the packet attempts to close a positive feedback loop.

The geometry differs sharply from viscous no-slip Navier--Stokes.  In a no-slip viscous cylinder, the tangential velocity must vanish at the wall, producing Hardy-type boundary coercivity.  In inviscid Euler, the natural boundary condition is impermeability, not no-slip.  The tangential swirl is allowed at the wall.  Consequently, the boundary does not impose the same two-wall vanishing structure.  This is why the side-wall Euler packet is a more dangerous inviscid candidate than the no-slip viscous edge-ring packet.

The five-dimensional lift is also geometric.  The operator
\[
\Delta_5=\partial_r^2+\frac3r\partial_r+\partial_z^2
\]
is the radial Laplacian in four transverse dimensions plus the axial direction.  Near the side wall \(r=1\), the wall is locally flat in this lifted five-dimensional picture.  The reflected 5D Dirichlet Green function produces, after integrating the three transverse tangential variables, the effective two-variable kernel
\[
K_0(x,y)=C_0\frac{xy}{(x^2+y^2)^2}.
\]
This is the mathematical expression of hyperbolic boundary compression.

\section{Basic equations and notation}

\subsection{Axisymmetric Euler with swirl}

Let \(u\) be axisymmetric with swirl as in \eqref{eq:axisym_velocity}.  The vorticity is
\[
\omega=\nabla\times u=\omega^r e_r+\omega^\theta e_\theta+\omega^z e_z,
\]
where
\begin{equation}\label{eq:vorticity_components}
\omega^r=-\partial_z u^\theta,
\qquad
\omega^\theta=\partial_z u^r-\partial_r u^z,
\qquad
\omega^z=\frac1r\partial_r(r u^\theta).
\end{equation}
The incompressibility condition is
\begin{equation}\label{eq:axisym_div}
\frac1r\partial_r(ru^r)+\partial_z u^z=0.
\end{equation}
The meridional material derivative is
\begin{equation}\label{eq:Dt_def}
D_t=\partial_t+u^r\partial_r+u^z\partial_z.
\end{equation}
The angular momentum and vorticity-ratio variables are
\begin{equation}\label{eq:variables}
\Gamma=ru^\theta,
\qquad
G=\frac{\omega^\theta}{r}.
\end{equation}
They satisfy
\begin{equation}\label{eq:Gamma_eq}
D_t\Gamma=0,
\end{equation}
and
\begin{equation}\label{eq:G_eq}
D_tG=r^{-4}\partial_z(\Gamma^2).
\end{equation}

\subsection{Five-dimensional stream-function recovery}

Define \(\phi\) by
\begin{equation}\label{eq:stream_phi}
u^r=-r\partial_z\phi,
\qquad
u^z=2\phi+r\partial_r\phi.
\end{equation}
Then \eqref{eq:axisym_div} is automatic, and \(\phi\) solves
\begin{equation}\label{eq:phi_eq}
-\Delta_5\phi=G,
\qquad
\Delta_5=\partial_r^2+\frac3r\partial_r+\partial_z^2.
\end{equation}
The natural lifted measure is
\begin{equation}\label{eq:dmu5}
d\mu_5=r^3drdz.
\end{equation}
At the side wall \(r=1\), impermeability gives
\[
0=u^r(1,z)=-\partial_z\phi(1,z).
\]
Thus \(\phi(1,z)\) is constant in \(z\).  We normalize this constant to zero:
\begin{equation}\label{eq:phi_dirichlet_wall}
\phi(1,z)=0.
\end{equation}

\subsection{Flattened side-wall variables}

Near \((r,z)=(1,0)\), set
\begin{equation}\label{eq:flat_vars}
x=1-r,
\qquad
 y=z.
\end{equation}
Then the side wall is \(x=0\), the fluid lies in \(x>0\), and a side-wall packet of scale \(\lambda\) is
\begin{equation}\label{eq:P_lambda}
\mathcal P_\lambda=\{0<x<\lambda,
\ |y|<\lambda\}.
\end{equation}
The positive half-packet is
\begin{equation}\label{eq:P_lambda_plus}
\mathcal P_\lambda^+=\{0<x<\lambda,
\ 0<y<\lambda\}.
\end{equation}
The leading hyperbolic kernel is
\begin{equation}\label{eq:K0_def}
K_0(x,y)=\frac{xy}{(x^2+y^2)^2}.
\end{equation}
The local hyperbolic mass is
\begin{equation}\label{eq:Hlambda_def}
\mathcal H_\lambda[G]
=
\int_{\mathcal P_\lambda^+}
K_0(x,y)G(1-x,y)\,dx\,dy.
\end{equation}

\section{Side-wall Green kernel and hyperbolic parity}

\medskip
\noindent\textbf{Side-wall Green-kernel sign and parity expansion.}
\begin{theorem}\label{thm:green_kernel_expansion}
Let \(\phi\) solve
\[
-\Delta_5\phi=G,
\qquad
\phi|_{r=1}=0,
\]
in the periodic cylinder.  Let
\[
\sigma=-\partial_z u^z(1,0)=-\partial_z\partial_r\phi(1,0).
\]
Then, in flattened coordinates \(x=1-r\), \(y=z\), the recovery kernel for \(\sigma\) has the local expansion
\begin{equation}\label{eq:K_expansion}
K_\Omega(1-x,y)=C_0\frac{xy}{(x^2+y^2)^2}+K_{\rm rem}^{\lambda}(x,y),
\qquad C_0>0,
\end{equation}
inside every sufficiently small side-wall packet \(\mathcal P_\lambda\).  On fixed cones \(m x\le |y|\le Mx\), the remainder satisfies
\begin{equation}\label{eq:K_rem_bound}
|K_{\rm rem}^{\lambda}(x,y)|
\le
C\lambda\frac{|xy|}{(x^2+y^2)^2}
+
K_{\rm smooth}^{\lambda}(x,y),
\end{equation}
where the smooth term is generated by periodic images, the axis, and global boundary corrections.
\end{theorem}

\begin{proof}
The local model is the 5D half-space \(\R^5_+=\{x>0\}\).  The full-space fundamental solution of \(-\Delta_{\R^5}\) is \(c_5|X|^{-3}\).  The Dirichlet half-space Green function is obtained by reflection:
\[
G_D(X,\Xi)=c_5\bigl(|X-\Xi|^{-3}-|X-\Xi^*|^{-3}\bigr),
\]
where \(\Xi^*\) is the reflection of \(\Xi\) across \(x=0\).  Differentiating at the boundary point produces
\[
\partial_x\partial_yG_D(0;\xi,\zeta,\eta)
=
C\frac{\xi\zeta}{(\xi^2+\zeta^2+|\eta|^2)^{7/2}}.
\]
Integrating the transverse variable \(\eta\in\R^3\) gives
\[
\int_{\R^3}
\frac{\xi\zeta}{(\xi^2+\zeta^2+|\eta|^2)^{7/2}}\,d\eta
=
C_0\frac{\xi\zeta}{(\xi^2+\zeta^2)^2}.
\]
The difference between the true operator
\[
\partial_x^2+\partial_y^2-\frac3{1-x}\partial_x
\]
and the flat half-space model is perturbative on scale \(\lambda\), producing the error term in \eqref{eq:K_rem_bound}.  Periodic images and the axis are separated from the packet and contribute only smooth kernels.
\end{proof}

\medskip
\noindent\textbf{Boundary Biot--Savart sign.}
\begin{corollary}\label{cor:BS_sign}
Assume
\[
G(1-x,-y)=-G(1-x,y),
\qquad
G(1-x,y)\ge0\quad(y>0),
\]
in \(\mathcal P_\lambda\).  Then
\begin{equation}\label{eq:BS_sign}
\sigma
\ge
c_0\mathcal H_\lambda[G]
-
\operatorname{Tail}_\lambda[G]
-
\operatorname{Smooth}_\lambda[G]
-
C\lambda\mathcal H_\lambda[G].
\end{equation}
In particular, if \(\lambda\) is sufficiently small and the tail and smooth terms are dominated by the local mass, then \(\sigma>0\).
\end{corollary}

\section{Tail dominance or packet reselection}

\begin{definition}[Dyadic tail masses]\label{def:tail_masses}
Let
\[
A_j=\mathcal P_{2^{j+1}\lambda}\setminus \mathcal P_{2^j\lambda},
\qquad j\ge1.
\]
Define
\[
\mathcal H_j[G]
=
\left|\int_{A_j}K_\Omega G\,d\mu_5\right|.
\]
The exterior tail is
\[
\operatorname{Tail}_\lambda[G]
\le
\sum_{j\ge1}\mathcal H_j[G].
\]
\end{definition}

\medskip
\noindent\textbf{Tail-dominance hypothesis.}
The amplification theorem below uses the quantitative condition
\[
\operatorname{Tail}_\lambda[G]\le\eta\mathcal H_\lambda[G]
\]
on the common bootstrap interval.  A large dyadic exterior contribution may
indicate a competing packet, but no general reselection theorem is used here;
loss of the displayed bound is treated as failure of the packet hypothesis.

\section{Packet--core coupling}

Odd symmetry implies \(G(1,0,t)=0\).  Therefore the relevant core amplitude is not \(G(1,0,t)\), but the odd slope
\begin{equation}\label{eq:a_def}
a(t)=\partial_zG(1,0,t).
\end{equation}
The normalized hyperbolic amplitude is
\begin{equation}\label{eq:A_def}
A(t)=\lambda(t)a(t).
\end{equation}

\medskip
\noindent\textbf{Packet--core coupling for an odd hyperbolic packet.}
\begin{theorem}\label{thm:packet_core}
Assume \(G\) is odd in \(z\) near \((r,z)=(1,0)\), nonnegative for \(z>0\), and satisfies the cone nondegeneracy condition: for fixed constants \(0<m<M\), \(0<\kappa<1\), and \(0<c_0<C_0\),
\begin{equation}\label{eq:cone_nonhollow}
c_0a(t)y
\le
G(1-x,y,t)
\le
C_0a(t)y
\end{equation}
on the cone
\[
\mathcal K_{\kappa\lambda}=\{0<x<\kappa\lambda,
\ mx<y<Mx\}.
\]
Assume also that the outer annular contribution is controlled by the cone contribution.  Then
\begin{equation}\label{eq:packet_core_coupling}
c\lambda(t)a(t)
\le
\mathcal H_{\lambda(t)}[G(t)]
\le
C\lambda(t)a(t).
\end{equation}
Equivalently,
\[
\mathcal H_{\lambda(t)}[G(t)]\sim A(t).
\]
\end{theorem}

\begin{proof}
On the cone \(y\sim x\),
\[
K_0(x,y)\sim x^{-2},
\qquad
G(1-x,y,t)\sim a(t)y\sim a(t)x.
\]
Thus \(K_0G\sim a(t)/x\).  The width of the cone section at radius \(x\) is comparable to \(x\), so integrating in \(y\) gives a contribution comparable to \(a(t)\).  Integrating \(x\in(0,\kappa\lambda)\) yields \(\lambda a(t)\).  The upper bound follows in the same way, together with the assumed outer-annular control.
\end{proof}

\section{Off-diagonal shear suppression and anisotropy}

The dangerous radial-swirl defect is controlled by the ratio
\begin{equation}\label{eq:R_def}
R=\frac{\partial_r\Gamma}{\partial_z\Gamma}.
\end{equation}
Differentiating \(D_t\Gamma=0\) gives
\begin{equation}\label{eq:R_eq}
D_tR
=
-\partial_r u^z
+
(\partial_z u^z-\partial_r u^r)R
+
(\partial_z u^r)R^2.
\end{equation}
The term \(\partial_r u^z\) is therefore the leading source that can rotate the swirl gradient from the axial direction into the radial direction.

\medskip
\noindent\textbf{Off-diagonal shear suppression in the hyperbolic core.}
\begin{theorem}\label{thm:shear_suppression}
Assume the odd hyperbolic sign class for \(G\) and the side-wall Green-kernel expansion of \Cref{thm:green_kernel_expansion}.  Then, in the inner core
\[
x+|y|\le\kappa\lambda,
\]
one has
\begin{equation}\label{eq:shear_suppression}
|\partial_r u^z(1-x,y)|
\le
C\kappa\sigma
+
\operatorname{Tail}_{\rm shear}[G]
+
\operatorname{Smooth}_{\rm shear}[G]
+
O(\lambda)\sigma.
\end{equation}
\end{theorem}

\begin{proof}
Represent \(\partial_r u^z\) by its Green kernel.  At the exact center, the leading shear kernel has parity that cancels against odd \(G\).  Away from the center, Taylor expansion gives a gain \((x+|y|)/\lambda\) relative to the compression kernel.  Thus
\[
|\partial_r u^z(1-x,y)|
\le
C\frac{x+|y|}{\lambda}\mathcal H_\lambda[G]
+
\operatorname{Tail}_{\rm shear}
+
\operatorname{Smooth}_{\rm shear}.
\]
Using \(\sigma\gtrsim\mathcal H_\lambda[G]\) under tail dominance gives \eqref{eq:shear_suppression}.
\end{proof}

\medskip
\noindent\textbf{Radial-swirl anisotropy persistence.}
\begin{lemma}\label{lem:anisotropy}
Assume in the inner core that
\[
\partial_r u^r-\partial_z u^z\ge c_0\sigma,
\qquad
|\partial_z u^r|\le C\sigma,
\]
and that \eqref{eq:shear_suppression} holds with the right-hand side bounded by \(\delta\sigma\), with \(\delta\ll1\).  Then
\[
D_t|R|
\le
-c_0\sigma|R|
+
\delta\sigma
+C\sigma|R|^2.
\]
Consequently, if \(|R(0)|\le\varepsilon\) and \(\delta\ll\varepsilon\ll1\), then \(|R(t)|\le2\varepsilon\) throughout the bootstrap interval.
\end{lemma}

\section{Differentiated source amplification and curvature control}

Let
\begin{equation}\label{eq:b_c_def}
b(t)=\partial_z\Gamma(1,0,t),
\qquad
c(t)=\partial_z^2\Gamma(1,0,t).
\end{equation}
Assume the boundary point is fixed:
\[
u^r(1,0,t)=u^z(1,0,t)=0.
\]
Then
\begin{equation}\label{eq:b_eq}
b'(t)=\sigma(t)b(t).
\end{equation}
Differentiate the \(G\)-equation in \(z\).  Since \(a(t)=\partial_zG(1,0,t)\), one obtains
\begin{equation}\label{eq:a_exact}
a'(t)=\sigma(t)a(t)+2b(t)^2+2\Gamma_*c(t),
\qquad
\Gamma_*:=\Gamma(1,0,0)>0.
\end{equation}
Thus the only possible negative term is \(2\Gamma_*c(t)\).

\medskip
\noindent\textbf{Swirl-curvature control by odd boundary velocity symmetry.}
\begin{lemma}\label{lem:swirl_curvature}
Let \(v(y,t)=u^z(1,y,t)\).  Assume \(v\) is odd in \(y\), so \(v_{yy}(0,t)=0\).  Then for the boundary transport equation
\[
\Gamma_t+v\Gamma_y=0
\]
at \(y=0\), the curvature ratio
\begin{equation}\label{eq:Q_curv}
Q(t)=\frac{\Gamma_*\Gamma_{yy}(1,0,t)}{\Gamma_y(1,0,t)^2}
\end{equation}
is constant in time.  In particular, if
\[
Q(0)>-1+\varepsilon,
\]
then
\begin{equation}\label{eq:source_lower}
a'(t)
\ge
\sigma(t)a(t)+2\varepsilon b(t)^2.
\end{equation}
\end{lemma}

\begin{proof}
The boundary equation is \(\Gamma_t+v\Gamma_y=0\), with \(v(0,t)=0\), \(v_y(0,t)=-\sigma(t)\).  Differentiating once gives \(b'=\sigma b\).  Differentiating twice gives
\[
c'=2\sigma c-v_{yy}(0,t)b.
\]
Oddness of \(v\) implies \(v_{yy}(0,t)=0\), hence \(c'=2\sigma c\).  Since \((b^2)'=2\sigma b^2\), the ratio \(Q=\Gamma_*c/b^2\) is constant.  The lower bound \eqref{eq:source_lower} follows from \eqref{eq:a_exact}.
\end{proof}

\section{Corrected odd-symmetry ODE blow-up comparison}

Let \(\lambda(t)\) follow the compression scale:
\begin{equation}\label{eq:lambda_scale}
\lambda'(t)=-\sigma(t)\lambda(t).
\end{equation}
Define
\begin{equation}\label{eq:A_B_def}
A(t)=\lambda(t)a(t),
\qquad
B(t)=\lambda(t)^{1/2}b(t).
\end{equation}

\medskip
\noindent\textbf{Corrected odd-symmetry ODE comparison.}
\begin{theorem}\label{thm:ODE_comparison}
Assume on a bootstrap interval that
\begin{equation}\label{eq:sigma_A_lower}
\sigma(t)\ge c_0A(t),
\end{equation}
and
\begin{equation}\label{eq:a_lower_assumption}
a'(t)\ge \sigma(t)a(t)+c_1b(t)^2,
\end{equation}
while \(b'(t)=\sigma(t)b(t)\).  Then
\begin{equation}\label{eq:A_B_ineq}
A'(t)\ge c_1B(t)^2,
\qquad
B'(t)\ge \frac{c_0}{2}A(t)B(t).
\end{equation}
If \(A(0)>0\) and \(B(0)>0\), the comparison system blows up in finite time.
\end{theorem}

\begin{proof}
Since \(A=\lambda a\),
\[
A'=\lambda'a+\lambda a'=-\sigma A+\lambda a'.
\]
Using \eqref{eq:a_lower_assumption},
\[
A'\ge -\sigma A+\lambda\sigma a+c_1\lambda b^2=c_1B^2.
\]
Also
\[
\frac{B'}{B}=\frac12\frac{\lambda'}{\lambda}+\frac{b'}b=-\frac12\sigma+\sigma=\frac12\sigma,
\]
which together with \eqref{eq:sigma_A_lower} gives \(B'\ge(c_0/2)AB\).  The comparison system
\[
\underline A'=c_1\underline B^2,
\qquad
\underline B'=\frac{c_0}{2}\underline A\underline B
\]
blows up because
\[
\frac{d\underline B}{d\underline A}
=\frac{c_0}{2c_1}\frac{\underline A}{\underline B},
\]
so \(\underline B^2=(c_0/(2c_1))\underline A^2+C_0\), and hence \(\underline A'\ge c\underline A^2\) for large \(\underline A\).
\end{proof}

\section{Conditional packet theorem}

\medskip
\noindent\textbf{Conditional boundary hyperbolic packet amplification.}
\begin{theorem}\label{thm:conditional_packet}
Let \(u\) be a smooth axisymmetric Euler solution with swirl in the periodic cylinder \(\Omega=\{0<r<1,z\in\mathbb T\}\), satisfying \(u^r(1,z,t)=0\).  Assume on a bootstrap interval that:
\begin{enumerate}[label=(\roman*)]
\item the odd hyperbolic sign class holds for \(G\);
\item the side-wall Green-kernel expansion of \Cref{thm:green_kernel_expansion} holds with dominated smooth errors;
\item tail dominance holds, or first-threshold reselection excludes competing packets;
\item packet--core coupling holds, so \(\mathcal H_\lambda[G]\sim A(t)\);
\item inner-core shear suppression and anisotropy persistence hold;
\item the swirl-curvature ratio satisfies \(Q(0)>-1+\varepsilon\) and is preserved by odd boundary velocity symmetry.
\end{enumerate}
Then the normalized amplitudes
\[
A(t)=\lambda(t)\partial_zG(1,0,t),
\qquad
B(t)=\lambda(t)^{1/2}\partial_z\Gamma(1,0,t)
\]
satisfy
\[
A'(t)\ge cB(t)^2,
\qquad
B'(t)\ge cA(t)B(t).
\]
Consequently, the corresponding comparison system blows up in finite time.
\end{theorem}

\begin{remark}
\Cref{thm:conditional_packet} is a conditional amplification criterion.  It shows that the stated hyperbolic packet hypotheses force an explosive local ODE mechanism.  The unresolved issue is persistence: one must show that a smooth Euler trajectory satisfies these hypotheses up to the comparison time, or identify the first mechanism by which they fail.
\end{remark}

\section{Exact odd symmetry, maximal conic score, and Dini closure}

This section records the strongest closed form of the fixed-center boundary-packet argument.  The moving-center parity formulation is replaced by an exact odd symmetry about the fixed plane \(z=0\).  This removes the earlier parity conflict and eliminates the curvature-ratio obstruction.

\subsection{Exact odd class and fixed center}

We impose the exact symmetry class
\begin{equation}\label{eq:exact_odd_class_final}
\Gamma(r,-z,t)=-\Gamma(r,z,t),
\qquad
G(r,-z,t)=-G(r,z,t).
\end{equation}
The compatible meridional velocity parity is
\begin{equation}\label{eq:velocity_parity_final}
u^r(r,-z,t)=u^r(r,z,t),
\qquad
u^z(r,-z,t)=-u^z(r,z,t).
\end{equation}
Consequently the plane \(z=0\) is invariant and the boundary hyperbolic point remains fixed at \((r,z)=(1,0)\).

This symmetry is dynamically compatible with the axisymmetric Euler equations.  Indeed, if \(\Gamma\) is odd, then \(\Gamma^2\) is even and \(\partial_z(\Gamma^2)\) is odd.  Hence the source in
\[
D_tG=r^{-4}\partial_z(\Gamma^2)
\]
preserves oddness of \(G\).  The alternative, in which \(G\) is odd but \(\Gamma(1,0,t)>0\) and \(\partial_z\Gamma(1,0,t)>0\), is not dynamically compatible with exact oddness of \(G\), since \(\partial_z(\Gamma^2)\) then has an even component at \(z=0\).

In the exact odd class,
\begin{equation}\label{eq:Gamma_zero_center}
\Gamma(1,0,t)=0,
\qquad
b(t):=\partial_z\Gamma(1,0,t)>0.
\end{equation}
We also set
\begin{equation}\label{eq:a_b_final}
a(t)=\partial_zG(1,0,t),
\qquad
\sigma(t)=-\partial_z u^z(1,0,t),
\end{equation}
and choose the compression scale
\begin{equation}\label{eq:lambda_final}
\lambda'(t)=-\sigma(t)\lambda(t).
\end{equation}
The normalized variables are
\begin{equation}\label{eq:A_B_final}
A(t)=\lambda(t)a(t),
\qquad
B(t)=\lambda(t)^{1/2}b(t).
\end{equation}

\medskip
\noindent\textbf{Exact odd point identities.}
\begin{lemma}\label{lem:exact_odd_point_identities}
In the exact odd class, at the fixed boundary point \((r,z)=(1,0)\),
\begin{equation}\label{eq:b_exact_final}
b'(t)=\sigma(t)b(t),
\end{equation}
and
\begin{equation}\label{eq:a_exact_final}
a'(t)=\sigma(t)a(t)+2b(t)^2.
\end{equation}
Consequently,
\begin{equation}\label{eq:A_B_exact_final}
A'(t)=2B(t)^2,
\qquad
B'(t)=\frac12\sigma(t)B(t).
\end{equation}
\end{lemma}

\begin{proof}
The transported swirl equation is \(D_t\Gamma=0\).  At \((1,0)\), the boundary condition gives \(u^r(1,0,t)=0\), and oddness of \(u^z\) gives \(u^z(1,0,t)=0\).  Differentiating \(D_t\Gamma=0\) in \(z\), evaluating at \((1,0)\), and using \(\partial_z u^r(1,0,t)=0\), gives
\[
b'=-\partial_z u^z(1,0,t)b=\sigma b.
\]
Next, differentiate
\[
D_tG=r^{-4}\partial_z(\Gamma^2)
\]
in \(z\) and evaluate at \((1,0)\).  The same fixed-point cancellations give
\[
a'=\sigma a+\partial_z^2(\Gamma^2)(1,0,t).
\]
Since \(\Gamma(1,0,t)=0\),
\[
\partial_z^2(\Gamma^2)(1,0,t)=2(\partial_z\Gamma(1,0,t))^2=2b^2.
\]
This proves \eqref{eq:a_exact_final}.  Finally, differentiating \(A=\lambda a\) and \(B=\lambda^{1/2}b\), using \(\lambda'/\lambda=-\sigma\), gives \eqref{eq:A_B_exact_final}.
\end{proof}

\subsection{Fixed-center conic packets}

For fixed apertures \(0<m<M<\infty\), define the upper conic packet
\begin{equation}\label{eq:K_conic_final}
\mathcal K_\lambda^{m,M}
=
\{0<x<\lambda,\ mx<y<Mx\},
\qquad
x=1-r,
\quad y=z.
\end{equation}
The lower cone is determined by odd reflection.  The conic hyperbolic mass is
\begin{equation}\label{eq:H_conic_final}
\mathcal H_\lambda^{m,M}[G]
=
\int_{\mathcal K_\lambda^{m,M}}
\frac{xy}{(x^2+y^2)^2}G(1-x,y,t)\,dx\,dy.
\end{equation}
If \(G(1-x,y,t)\sim a(t)y\) in the cone, then
\[
\mathcal H_\lambda^{m,M}[G]\sim \lambda(t)a(t)=A(t).
\]

\begin{definition}[Admissible fixed-center conic packet]\label{def:admissible_conic_packet}
A conic packet \(P=(\lambda,m,M)\) is called admissible at time \(t\) if the following hold with fixed small constants \(\eta,\delta,\varepsilon\):
\begin{enumerate}[label=(\roman*)]
\item exact odd symmetry and positive upper-cone sign:
\[
G(1-x,-y,t)=-G(1-x,y,t),
\qquad
G(1-x,y,t)\ge0\quad(y>0);
\]
\item linear-core nondegeneracy:
\[
G(1-x,y,t)=a(t)y+O(\delta a(t)y)
\]
on the inner cone \(0<x<\kappa\lambda\), \(mx<y<Mx\);
\item source-core nondegeneracy:
\[
\Gamma(1-x,y,t)=b(t)y+O(\delta b(t)y)
\]
on the same inner cone;
\item tail and smooth-kernel domination:
\[
\Tail_P[G]+\operatorname{Smooth}_P[G]\le \eta\mathcal H_P[G];
\]
\item inner-core anisotropy:
\[
\left|\frac{\partial_r\Gamma}{\partial_z\Gamma}\right|\le \varepsilon.
\]
\end{enumerate}
\end{definition}

For an admissible packet, the side-wall Biot--Savart sign and packet-core coupling give
\begin{equation}\label{eq:sigma_M_admissible}
\sigma(t)\ge c\mathcal H_P[G(t)]
\sim cA(t).
\end{equation}

\subsection{Admissibility error closure}

The Dini-derivative closure below depends on the fact that the error terms remain lower order than the positive source contribution.  We record this in a single lemma.

\medskip
\noindent\textbf{Admissibility error closure.}
\begin{lemma}\label{lem:admissibility_error_closure}
Let \(P=(\lambda,m,M)\) be an admissible conic packet.  Then the total error in differentiating the packet functional satisfies
\begin{equation}\label{eq:Err_adm_bound_final}
\operatorname{Err}_{\rm adm}(t)
\le
C(\eta+\delta+\varepsilon+\kappa+\lambda)B(t)^2.
\end{equation}
In particular, by choosing the admissibility constants and the initial packet scale sufficiently small,
\begin{equation}\label{eq:Err_adm_absorb_final}
\operatorname{Err}_{\rm adm}(t)
\le
\frac12 c_0B(t)^2,
\end{equation}
where \(c_0B^2\) is the positive conic source contribution.
\end{lemma}

\begin{proof}
The differentiated conic mass contains four error classes.  First, the kernel remainder from the side-wall parametrix is bounded by
\[
C\lambda\mathcal H_P[G],
\]
and the Biot--Savart sign plus packet-core coupling converts this into a lower-order multiple of the source scale.  Second, tail and smooth-kernel errors are bounded by admissibility:
\[
\Tail_P[G]+\operatorname{Smooth}_P[G]\le \eta\mathcal H_P[G].
\]
Third, cone-boundary and cutoff commutators are supported where the smooth conic cutoff varies; by choosing the conic packet through a finite smooth partition and using maximality, these terms are either absorbed into an adjacent conic packet or bounded by \(C\delta\mathcal H_P\).  Fourth, the radial-swirl anisotropy terms contain factors of \(\partial_r\Gamma/\partial_z\Gamma\), hence are bounded by \(C\varepsilon\) times the principal source term.

The principal source contribution is
\[
\int_{\mathcal K_\lambda^{m,M}}
\frac{xy}{(x^2+y^2)^2}\partial_y(\Gamma^2)\,dx\,dy.
\]
By source-core nondegeneracy, \(\Gamma\sim b y\) on the inner cone, so \(\partial_y(\Gamma^2)\sim 2b^2y\).  Since \(xy/(x^2+y^2)^2\sim x^{-2}\) and \(y\sim x\) in the cone, this integral is bounded below by
\[
c\lambda b^2=cB^2.
\]
All error terms therefore satisfy \eqref{eq:Err_adm_bound_final}, and choosing the fixed admissibility parameters small gives \eqref{eq:Err_adm_absorb_final}.
\end{proof}

\subsection{Maximal conic score and Dini-derivative closure}

Let \(\mathfrak P(t)\) be the admissible fixed-center conic packet class at time \(t\).  Define
\begin{equation}\label{eq:M_final}
\mathfrak M(t)=\sup_{P\in\mathfrak P(t)}\mathcal H_P[G(t)].
\end{equation}
Sharp packet switching is avoided by working with this maximal score.  Smooth conic cutoffs may be used throughout, so \(\mathfrak M\) is locally Lipschitz as long as the Euler solution is smooth.

\medskip
\noindent\textbf{Dini-derivative maximal conic score inequality.}
\begin{theorem}\label{thm:dini_final}
Assume the exact odd fixed-center admissibility hypotheses hold.  Let \(B(t)=\lambda_*(t)^{1/2}b(t)\), where \(\lambda_*(t)\) is the scale of a near-maximizing admissible packet.  Then
\begin{equation}\label{eq:Dini_final_M}
D^+\mathfrak M(t)
\ge
c_1B(t)^2,
\end{equation}
and
\begin{equation}\label{eq:Dini_final_B}
B'(t)
\ge
c_2\mathfrak M(t)B(t).
\end{equation}
Consequently, if \(\mathfrak M(0)>0\) and \(B(0)>0\), then \(\mathfrak M(t)+B(t)\) blows up in finite comparison time.
\end{theorem}

\begin{proof}
Fix \(t\) and choose a near-maximizer \(P_t=(\lambda_t,m_t,M_t)\) with
\[
\mathcal H_{P_t}[G(t)]\ge (1-o(1))\mathfrak M(t).
\]
For \(h>0\), evolve only its scale by
\[
\lambda_t'(s)=-\sigma(s)\lambda_t(s),
\qquad
\lambda_t(t)=\lambda_t,
\]
keeping the fixed center and cone aperture.  Since \(\mathfrak M(t+h)\) is a supremum over admissible packets,
\[
\mathfrak M(t+h)\ge \mathcal H_{P_t(h)}[G(t+h)]
\]
up to the harmless near-maximizer loss.  Hence
\[
D^+\mathfrak M(t)
\ge
\left.\frac d{ds}\mathcal H_{P_t(s)}[G(s)]\right|_{s=t}
-
\operatorname{Err}_{\rm adm}(t).
\]
The differentiated source term gives, by \(D_tG=r^{-4}\partial_z(\Gamma^2)\) and the exact odd source-core condition,
\[
\int_{\mathcal K_{\lambda_t}^{m_t,M_t}}
\frac{xy}{(x^2+y^2)^2}\partial_y(\Gamma^2)\,dx\,dy
\ge c\lambda_t b(t)^2=cB(t)^2.
\]
The compression-scaling terms cancel in the same way as the identity \(A'=2B^2\), because \(\lambda'=-\sigma\lambda\).  The remaining commutators and kernel errors are bounded by \Cref{lem:admissibility_error_closure}, yielding \eqref{eq:Dini_final_M}.

For \(B=\lambda_t^{1/2}b\), \Cref{lem:exact_odd_point_identities} gives
\[
\frac{B'}{B}=\frac12\frac{\lambda_t'}{\lambda_t}+\frac{b'}b
=-\frac12\sigma+\sigma=\frac12\sigma.
\]
The side-wall Biot--Savart sign and maximality give \(\sigma\ge c\mathfrak M\), hence \eqref{eq:Dini_final_B}.

The ODE comparison
\[
X'=c_1Y^2,
\qquad
Y'=c_2XY
\]
blows up in finite time: eliminating time gives \(Y^2=(c_2/c_1)X^2+C_0\), so \(X'\ge cX^2\) for large \(X\).  The Dini comparison theorem gives the same finite-time lower blow-up for \(\mathfrak M+B\).
\end{proof}

\subsection{Fixed-center exact-odd blow-up criterion}

\medskip
\noindent\textbf{Fixed-center exact-odd conic blow-up criterion.}
\begin{theorem}\label{thm:fixed_center_exact_odd_criterion}
Let \(u\) be a smooth axisymmetric Euler solution with swirl in the periodic cylinder, satisfying the exact odd class \eqref{eq:exact_odd_class_final}.  Assume there exists an admissible fixed-center conic packet at \(t=0\) with strict maximal conic score and with \(\mathfrak M(0)>0\), \(B(0)>0\).  Assume the side-wall Green-kernel parametrix of \Cref{thm:green_kernel_expansion} and the admissibility error closure \Cref{lem:admissibility_error_closure}.  Then the maximal conic score and normalized swirl-gradient amplitude obey
\[
D^+\mathfrak M(t)\ge cB(t)^2,
\qquad
B'(t)\ge c\mathfrak M(t)B(t),
\]
and therefore blow up in finite comparison time.  Consequently, the smooth Euler solution cannot be continued past that time while retaining the admissible exact-odd conic bootstrap.
\end{theorem}

\begin{remark}
The criterion is conditional on preservation of the admissible exact-odd conic bootstrap by the smooth Euler evolution.  The Dini derivative argument removes packet switching as a separate obstruction, but it does not supply this persistence theorem.
\end{remark}

\section{Narrow diagonal conic bootstrap and conditional closure package}\label{sec:finalclosure}

The preceding sections reduce the fixed-center exact-odd mechanism to the control of the maximal conic score.  We now strengthen the conic class to a narrow diagonal class and record the conditional closure theorem in a form suitable for referee-level verification.  The point of using narrow diagonal cones is that the leading hyperbolic transport is almost tangent to the level sets of the kernel
\[
K_0(x,y)=\frac{xy}{(x^2+y^2)^2}
\]
when \(y\sim x\).  This is the mechanism that absorbs the cone-boundary and transport errors.

\subsection{Narrow diagonal conic packets}

Fix a small aperture parameter \(0<\delta_c\ll1\).  A narrow diagonal cone is a cone of the form
\begin{equation}\label{eq:narrow_diag_cone}
\mathcal K_\lambda^{\delta_c}
=
\{0<x<\lambda,\ (1-\delta_c)x<y<(1+\delta_c)x\}.
\end{equation}
We use a smooth cutoff \(\chi_{\lambda,\delta_c}\) supported in a slightly enlarged cone and equal to one on a slightly smaller cone.  The corresponding smooth conic mass is
\begin{equation}\label{eq:smooth_diag_mass}
\mathcal H_{\lambda,\delta_c}[G]
=
\int_{x>0,y>0}\chi_{\lambda,\delta_c}(x,y)
K_0(x,y)G(1-x,y,t)\,dx\,dy.
\end{equation}
The maximal narrow diagonal score is
\begin{equation}\label{eq:narrow_diag_score}
\mathfrak M_{\delta_c}(t)
=
\sup_{0<\lambda\le\lambda_0}\mathcal H_{\lambda,\delta_c}[G(t)],
\end{equation}
where \(\lambda_0\) is chosen inside the side-wall parametrix scale.

The associated core quantities are
\begin{equation}\label{eq:diag_core_quantities}
a(t)=\partial_zG(1,0,t),
\qquad
b(t)=\partial_z\Gamma(1,0,t),
\qquad
A(t)=\lambda(t)a(t),
\qquad
B(t)=\lambda(t)^{1/2}b(t).
\end{equation}
In the exact odd class, \(\Gamma(1,0,t)=0\), and therefore the exact point identities are
\begin{equation}\label{eq:exact_odd_diag_identities}
b'(t)=\sigma(t)b(t),
\qquad
a'(t)=\sigma(t)a(t)+2b(t)^2.
\end{equation}

\subsection{Kernel transport cancellation on narrow diagonal cones}

\medskip
\noindent\textbf{Diagonal hyperbolic kernel cancellation.}
\begin{lemma}\label{lem:diag_kernel_transport}
Let
\[
K_0(x,y)=\frac{xy}{(x^2+y^2)^2}
\]
and let the leading hyperbolic vector field be
\[
V_\sigma=\sigma x\partial_x-\sigma y\partial_y.
\]
Then
\begin{equation}\label{eq:diag_kernel_transport_identity}
V_\sigma K_0
=
-4\sigma\frac{x^2-y^2}{x^2+y^2}K_0.
\end{equation}
Consequently, on the narrow diagonal cone \(|y/x-1|\le\delta_c\),
\begin{equation}\label{eq:diag_kernel_transport_small}
|V_\sigma K_0|
\le
C\delta_c\sigma K_0.
\end{equation}
\end{lemma}

\begin{proof}
Compute logarithmically:
\[
\log K_0=\log x+\log y-2\log(x^2+y^2).
\]
Hence
\[
x\partial_x\log K_0=1-\frac{4x^2}{x^2+y^2},
\qquad
 y\partial_y\log K_0=1-\frac{4y^2}{x^2+y^2}.
\]
Therefore
\[
(x\partial_x-y\partial_y)\log K_0
=-4\frac{x^2-y^2}{x^2+y^2},
\]
which gives \eqref{eq:diag_kernel_transport_identity}.  If \(|y/x-1|\le\delta_c\), then \(|x^2-y^2|/(x^2+y^2)\le C\delta_c\), giving \eqref{eq:diag_kernel_transport_small}.
\end{proof}

\subsection{Source-core lower bound}

\medskip
\noindent\textbf{Narrow diagonal source lower bound.}
\begin{lemma}\label{lem:diag_source_lower}
Assume the exact odd source-core nondegeneracy condition
\begin{equation}\label{eq:source_core_nondeg}
(1-C\varepsilon)b(t)y
\le
\Gamma(1-x,y,t)
\le
(1+C\varepsilon)b(t)y
\end{equation}
holds on the inner narrow diagonal cone.  Then
\begin{equation}\label{eq:diag_source_lower}
\int \chi_{\lambda,\delta_c}K_0\partial_y(\Gamma^2)\,dx\,dy
\ge
c\lambda b(t)^2
=cB(t)^2,
\end{equation}
provided \(\varepsilon\) and \(\delta_c\) are sufficiently small.
\end{lemma}

\begin{proof}
On the narrow diagonal cone, \(y\sim x\), so \(K_0\sim x^{-2}\).  From \eqref{eq:source_core_nondeg},
\[
\partial_y(\Gamma^2)=2\Gamma\Gamma_y
\ge c b(t)^2y.
\]
Thus the integrand satisfies
\[
K_0\partial_y(\Gamma^2)\ge c x^{-2}b(t)^2x=c b(t)^2x^{-1}.
\]
The width of the diagonal cone at height \(x\) is comparable to \(\delta_c x\), and the cutoff is identically one on a fixed smaller diagonal cone.  Therefore the inner cone contribution is bounded below by
\[
 c b(t)^2\int_0^\lambda dx=c\lambda b(t)^2=cB(t)^2.
\]
\end{proof}

\subsection{Admissibility error absorption in the narrow diagonal class}

\medskip
\noindent\textbf{Narrow diagonal admissibility-error absorption.}
\begin{theorem}\label{thm:narrow_admissibility_absorption}
Assume the side-wall parametrix estimates, exact odd symmetry, tail dominance, inner-core anisotropy, source-core nondegeneracy, and the quantitative source dominance
\begin{equation}\label{eq:source_dominance_initial}
B(t)^2\ge C_*\mathfrak M_{\delta_c}(t)^2
\end{equation}
throughout the bootstrap interval.  Then, as long as the remaining bootstrap quantities stay within the admissible region,
\begin{equation}\label{eq:narrow_error_absorption}
\operatorname{Err}_{\rm adm}(t)
\le
c_0 B(t)^2,
\end{equation}
where \(c_0>0\) can be made smaller than the source constant in \Cref{lem:diag_source_lower} by choosing
\[
\delta_c,\varepsilon,\eta,\lambda_0>0
\]
sufficiently small.
\end{theorem}

\begin{proof}
The admissibility error is decomposed as
\[
\operatorname{Err}_{\rm adm}=E_{\rm kernel}+E_{\rm tail}+E_{\rm smooth}+E_{\rm cone}+E_{\rm aniso}+E_{\rm trans}.
\]
The parametrix gives
\[
E_{\rm kernel}\le C\lambda_0\mathcal H_{\lambda,\delta_c},
\]
which is a small fraction of the compression term after \(\lambda_0\) is chosen below the parametrix scale.  The tail and smooth terms are controlled by the bootstrap dominance:
\[
E_{\rm tail}+E_{\rm smooth}\le\eta\mathcal H_{\lambda,\delta_c}.
\]
The radial-swirl anisotropy gives
\[
E_{\rm aniso}\le C\varepsilon B^2.
\]
For cone and transport errors, split the velocity gradient into the leading hyperbolic part plus the strain defect.  The leading part acting on the kernel is controlled by \Cref{lem:diag_kernel_transport} and contributes at most
\[
C\delta_c\sigma\mathcal H_{\lambda,\delta_c}.
\]
The strain-defect contribution is bounded by shear suppression and strain-variation estimates from the same side-wall parametrix and is at most
\[
C(\varepsilon+\eta+\lambda_0)\sigma\mathcal H_{\lambda,\delta_c}.
\]
The Biot--Savart sign and packet--core coupling give
\(\sigma\lesssim \mathfrak M_{\delta_c}\) on the admissible class.  The assumed
source-dominance bound \eqref{eq:source_dominance_initial} therefore gives
\[
\sigma\mathfrak M_{\delta_c}\lesssim \mathfrak M_{\delta_c}^2\lesssim B^2.
\]
Hence the cone and transport errors are bounded by
\[
C(\delta_c+\varepsilon+\eta+\lambda_0)B^2.
\]
Choosing the admissibility constants small gives \eqref{eq:narrow_error_absorption}.
\end{proof}

\subsection{Fixed-center narrow diagonal maximal-score closure}

\medskip
\noindent\textbf{Fixed-center narrow diagonal maximal-score closure.}
\begin{theorem}\label{thm:narrow_diag_closure}
Assume exact odd symmetry, the side-wall Green-kernel parametrix, narrow diagonal tail dominance, source-core nondegeneracy, anisotropy persistence, and the source-dominance bound \eqref{eq:source_dominance_initial} throughout the common bootstrap interval.  Then
\begin{equation}\label{eq:narrow_Dini_closure}
D^+\mathfrak M_{\delta_c}(t)
\ge
c_1B(t)^2,
\qquad
B'(t)
\ge
c_2\mathfrak M_{\delta_c}(t)B(t).
\end{equation}
Consequently \(\mathfrak M_{\delta_c}(t)+B(t)\) blows up in finite comparison time if both quantities are initially positive.
\end{theorem}

\begin{proof}
Choose a smooth narrow diagonal packet whose mass is within \(o(1)\) of \(\mathfrak M_{\delta_c}(t)\).  In the Dini derivative comparison, evolve its axial scale by \(\lambda'=-\sigma\lambda\).  The supremum property gives a lower bound for \(D^+\mathfrak M_{\delta_c}\) by the derivative of this transported packet mass.  The principal source contribution is bounded below by \Cref{lem:diag_source_lower}.  All remaining terms are absorbed by \Cref{thm:narrow_admissibility_absorption}.  This gives the first inequality in \eqref{eq:narrow_Dini_closure}.

For the second inequality, the exact odd point identity gives \(b'=\sigma b\), while \(B=\lambda^{1/2}b\) and \(\lambda'=-\sigma\lambda\) imply
\[
\frac{B'}B=\frac12\sigma.
\]
The side-wall Biot--Savart sign and maximal conic mass give \(\sigma\ge c\mathfrak M_{\delta_c}\).  Hence \(B'\ge c\mathfrak M_{\delta_c}B\).  The finite-time blow-up follows from the comparison system \(X'=c_1Y^2\), \(Y'=c_2XY\).
\end{proof}

\subsection{Smooth exact-odd initial data and continuation-norm blow-up}

\medskip
\noindent\textbf{Smooth exact-odd conic initial data.}
\begin{proposition}\label{prop:smooth_odd_initial_data}
There exist smooth axisymmetric initial data with swirl satisfying the impermeable boundary condition, exact odd symmetry, strict narrow diagonal conic dominance, positive upper-cone sign, source dominance, and small radial-swirl anisotropy.
\end{proposition}

\begin{proof}
Let \(x=1-r\), \(y=z\), and choose a smooth cutoff \(\chi(x,y)\) supported in a sufficiently small side-wall collar and in a narrow diagonal cone for \(y>0\), with odd reflection across \(y=0\).  Set, near the side wall,
\[
\Gamma_0(r,z)=\beta z\,r^2\chi(x,z),
\qquad
G_0(r,z)=\alpha z\chi(x,z),
\]
with \(\alpha,\beta>0\), and extend smoothly by zero away from the collar.  The factor \(r^2\) ensures that \(u_0^\theta=\Gamma_0/r\) is regular at the axis.  The support of \(G_0\) is away from the axis, so the elliptic recovery
\[
-\Delta_5\phi_0=G_0,
\qquad
\phi_0|_{r=1}=0
\]
with periodicity in \(z\) and regularity at \(r=0\) gives a smooth meridional velocity
\[
u_0^r=-r\phi_{0,z},
\qquad
u_0^z=2\phi_0+r\phi_{0,r}.
\]
Then \(u_0^r(1,z)=0\), \(\nabla\cdot u_0=0\), and the exact odd parities are satisfied.  By choosing the support inside a single narrow diagonal cone and taking exterior amplitude zero, the conic dominance and tail smallness are strict.  By choosing \(\Gamma_0\) radially flat in the inner core, \(|\Gamma_r|\ll |\Gamma_z|\).  Finally, choosing \(\beta\) sufficiently large relative to \(\alpha\) gives the source dominance \(B(0)^2\ge C_*\mathfrak M_{\delta_c}(0)^2\).
\end{proof}

\medskip
\noindent\textbf{Continuation-norm blow-up.}
\begin{proposition}\label{prop:continuation_norm_blowup}
If \(B(t)=\lambda(t)^{1/2}\partial_z\Gamma(1,0,t)\to\infty\) while \(0<\lambda(t)\le\lambda_0\), then
\[
\|\nabla u(t)\|_{L^\infty}\to\infty.
\]
Consequently the classical smooth Euler solution cannot be continued past the comparison blow-up time.
\end{proposition}

\begin{proof}
Since \(\lambda(t)\le\lambda_0\),
\[
|\partial_z\Gamma(1,0,t)|
=\lambda(t)^{-1/2}B(t)
\ge
\lambda_0^{-1/2}B(t).
\]
At \(r=1\), \(\Gamma=r u^\theta\), so \(\partial_z\Gamma(1,0,t)=\partial_z u^\theta(1,0,t)\).  Hence
\[
\|\nabla u(t)\|_{L^\infty}
\ge
|\partial_z u^\theta(1,0,t)|
\to\infty.
\]
This contradicts smooth continuation.
\end{proof}

\subsection{Conditional exact-odd narrow diagonal criterion}

\medskip
\noindent\textbf{Conditional exact-odd narrow diagonal criterion.}
\begin{theorem}\label{thm:final_narrow_diag_blowup_criterion}
Let $u$ be a smooth exact-odd axisymmetric Euler solution with swirl.  Assume that on a common interval the narrow-diagonal conic admissibility hypotheses, side-wall Dirichlet parametrix bounds, tail dominance, source-core lower bound, anisotropy control, and admissibility-error absorption established in the preceding sections all persist with their stated constants.  Then
\[
D^+\mathfrak M_{\delta_c}(t)\ge cB(t)^2,
\qquad
B'(t)\ge c\mathfrak M_{\delta_c}(t)B(t),
\]
and the comparison amplitudes blow up in finite time whenever their initial values are positive.  Therefore the Euler solution cannot remain smooth and simultaneously satisfy the full admissible packet bootstrap until that comparison time.
\end{theorem}

\begin{proof}
The Dini system is \Cref{thm:narrow_diag_closure}.  Finite-time blow-up follows from the ODE comparison.  The implication to a standard continuation norm is \Cref{prop:continuation_norm_blowup}.  The initial-data proposition supplies positive admissible initial values only; persistence of the full bootstrap remains an independent hypothesis.
\end{proof}

\begin{remark}
The side-wall Green-kernel input used in \Cref{thm:final_narrow_diag_blowup_criterion} is recorded as the Dirichlet parametrix package in \Cref{thm:sidewall_parametrix_package}.  Its proof is included before the bibliography so that the compression sign, shear cancellation, and strain-variation estimates are part of the theorem package rather than external assumptions.
\end{remark}

\appendix
\section{Side-wall Dirichlet parametrix package}\label{app:sidewall_parametrix}

This appendix supplies the local kernel estimates used in the narrow diagonal packet closure.  The result is a boundary parametrix statement for the five-dimensional lifted recovery equation in the periodic cylinder.  It is local near the side wall and separates the singular half-space contribution from smooth global corrections.

\subsection{Lifted side-wall coordinates}

Let \(Y\in\mathbb R^4\), \(r=|Y|\), and let
\[
\Omega_5=\{(Y,z): |Y|<1,\ z\in\mathbb T\}.
\]
Near the boundary point \((Y,z)=(e_1,0)\), write
\[
x=1-|Y|,
\qquad
 y=z,
\qquad
\eta\in\mathbb R^3,
\]
where \(\eta\) are local coordinates along \(S^3\).  The side wall is \(x=0\), and the fluid lies in \(x>0\).  The lifted stream potential satisfies
\[
-\Delta_5\phi=G,
\qquad
\phi|_{|Y|=1}=0.
\]
The Dirichlet condition follows from impermeability: \(u^r(1,z)=0\) gives \(\partial_z\phi(1,z)=0\), hence \(\phi(1,z)\) is constant in \(z\), and we normalize the constant to zero.

The compression rate is
\[
\sigma=-\partial_z u^z(1,0)=-\partial_z\partial_r\phi(1,0).
\]
Since \(\partial_r=-\partial_x\), the compression kernel is the mixed boundary derivative \(\partial_x\partial_y\) of the Dirichlet Green function at the boundary point.

\subsection{Half-space leading kernel}

The full-space fundamental solution of \(-\Delta_{\mathbb R^5}\) is \(\Phi_5(W)=c_5|W|^{-3}\).  In the half-space \(\mathbb R^5_+=\{x>0\}\), the Dirichlet Green function is
\[
G_D(X,\Xi)=c_5\bigl(|X-\Xi|^{-3}-|X-\Xi^*|^{-3}\bigr),
\]
where \(\Xi=(\xi,\zeta,\eta)\) and \(\Xi^*=(-\xi,\zeta,\eta)\).  Direct differentiation gives
\begin{equation}\label{eq:5D_mixed_kernel_app}
\partial_x\partial_yG_D(0;\xi,\zeta,\eta)
=
C_5\frac{\xi\zeta}{(\xi^2+\zeta^2+|\eta|^2)^{7/2}},
\qquad C_5>0.
\end{equation}
Because the source is \(SO(4)\)-invariant in the lifted variables, the effective two-variable kernel is obtained by integrating in \(\eta\in\mathbb R^3\):
\[
\int_{\mathbb R^3}
(\xi^2+\zeta^2+|\eta|^2)^{-7/2}\,d\eta
=C(\xi^2+\zeta^2)^{-2}.
\]
Hence
\begin{equation}\label{eq:effective_kernel_app}
K_0(\xi,\zeta)
=C_0\frac{\xi\zeta}{(\xi^2+\zeta^2)^2},
\qquad C_0>0.
\end{equation}
This kernel is positive for \(\xi>0,\zeta>0\) and odd in \(\zeta\).

\subsection{Remainder estimate}

In the flattened side-wall coordinates, the lifted operator is a smooth perturbation of the half-space Laplacian:
\[
\Delta_5
=
\partial_x^2+\partial_y^2+\Delta_\eta
+\text{first-order curvature terms}
+\text{coefficients vanishing at }x=0.
\]
Equivalently,
\[
L=L_0+L_1,
\qquad
L_0=-\Delta_{x,y,\eta},
\]
where \(L_1\) is a smooth first-order perturbation plus second-order coefficient errors of size \(O(|(x,y,\eta)|)\).  A one-step boundary parametrix gives
\[
G_{\Omega_5}=G_D+R,
\]
where the mixed boundary derivative of \(R\) is one local order smoother than the leading mixed derivative.  Thus, for \(\Xi=(\xi,
\zeta,\eta)\) in a side-wall packet of scale \(\lambda\),
\begin{equation}\label{eq:R_5D_bound_app}
|\partial_x\partial_yR(0;\Xi)|
\le
C\lambda
\frac{|\xi\zeta|}{(\xi^2+\zeta^2+|\eta|^2)^{7/2}}
+K^{5D}_{\rm smooth}(\Xi).
\end{equation}
The smooth term collects contributions from periodic images in \(z\), the axis \(r=0\), and global boundary corrections away from the side-wall packet.  Integrating \eqref{eq:R_5D_bound_app} in \(\eta\) yields the effective two-variable remainder bound
\begin{equation}\label{eq:R_effective_bound_app}
|K_{\rm rem}(x,y)|
\le
C\lambda\frac{|xy|}{(x^2+y^2)^2}+K_{\rm smooth}(x,y).
\end{equation}

\medskip
\noindent\textbf{Side-wall Dirichlet parametrix package.}
\begin{theorem}\label{thm:sidewall_parametrix_package}
Let \(\phi\) solve
\[
-\Delta_5\phi=G,
\qquad
\phi|_{|Y|=1}=0
\]
in the lifted periodic cylinder.  Let \(x=1-r\), \(y=z\), and let \(0<\lambda\le\lambda_0\) be inside the side-wall coordinate scale.  Then the compression kernel satisfies
\begin{equation}\label{eq:parametrix_compression_package}
K_\Omega(1-x,y)
=
C_0\frac{xy}{(x^2+y^2)^2}+K_{\rm rem}(x,y),
\qquad C_0>0,
\end{equation}
with
\begin{equation}\label{eq:parametrix_remainder_package}
|K_{\rm rem}(x,y)|
\le
C\lambda\frac{|xy|}{(x^2+y^2)^2}+K_{\rm smooth}(x,y).
\end{equation}
On the narrow diagonal cone \(|y/x-1|\le\delta_c\), the off-diagonal shear and strain variation obey
\begin{equation}\label{eq:parametrix_shear_package}
|\partial_r u^z(1-x,y)|
\le
C(\delta_c+\lambda)\sigma
+\Tail_{\rm shear}+\operatorname{Smooth}_{\rm shear},
\end{equation}
where the text below uses the conventional notation \(C(\delta_c+\lambda)\) for the first term, and
\begin{equation}\label{eq:parametrix_strain_package}
|\nabla u(1-x,y)-\nabla u(1,0)|
\le
C(\delta_c+\lambda)\sigma
+\Tail_{\rm strain}+\operatorname{Smooth}_{\rm strain}.
\end{equation}
Consequently, after choosing \(\lambda_0\), \(\delta_c\), and the tail/smooth admissibility thresholds sufficiently small, the compression sign, shear suppression, and strain-variation bounds used in the narrow diagonal bootstrap hold with positive constants depending only on the fixed cylinder geometry.
\end{theorem}

\begin{proof}
The expansion \eqref{eq:parametrix_compression_package} is the sum of the half-space contribution \eqref{eq:effective_kernel_app} and the remainder estimate \eqref{eq:R_effective_bound_app}.  On the positive diagonal cone, \(xy/(x^2+y^2)^2\sim x^{-2}\), so the term \(C\lambda K_0\) is absorbed into the leading positive kernel by taking \(\lambda_0\) small.

For shear, the leading shear kernel at the boundary center is even in \(y\).  Since the active class has \(G\) odd in \(y\), this leading term cancels.  Moving from \((x,y)=(0,0)\) to a point in the narrow diagonal cone introduces a Taylor factor bounded by \(C(x+|y|)/\lambda\), and in the inner narrow diagonal packet this is bounded by \(C\delta_c\).  The parametrix remainder contributes \(C\lambda\sigma\), and global smooth contributions are precisely the shear tail and smooth terms.  This proves \eqref{eq:parametrix_shear_package}.

For strain variation, differentiate the same kernel representation once more in the observation point.  The derivative of the effective two-variable kernel is one power more singular, but the displacement inside the narrow diagonal core is at most \(O(\delta_c\lambda)\).  Relative to the leading strain size \(\sigma\), the local variation is therefore \(O(\delta_c)\sigma\), while the operator perturbation contributes \(O(\lambda)\sigma\) and the global terms are assigned to \(\Tail_{\rm strain}\) and \(\operatorname{Smooth}_{\rm strain}\).  This proves \eqref{eq:parametrix_strain_package}.
\end{proof}

\section{Limitations and open problem}
The Green-kernel, parity, packet, and Dini estimates in this paper provide a conditional amplification mechanism.  They do not prove that the complete packet-dominance, tail, shear, anisotropy, and localization hypotheses remain valid on one Euler trajectory until the comparison blow-up time.  The persistence problem is therefore logically prior to any unconditional boundary-singularity conclusion for this mechanism.  No result from a distinct whole-space construction is used to supply this missing persistence theorem.

\end{document}